\documentclass[11pt]{article}

\usepackage[a4paper,margin=1.15in]{geometry}
\usepackage[T1]{fontenc}
\usepackage{lmodern}
\usepackage{microtype}
\usepackage{amsmath,amssymb,amsthm,mathtools}
\usepackage{enumitem}
\usepackage{aliascnt}
\usepackage[hidelinks]{hyperref}
\usepackage[nameinlink,capitalise,noabbrev]{cleveref}

\newcommand{\F}{\mathbb F}
\newcommand{\GL}{\operatorname{GL}}
\newcommand{\SL}{\operatorname{SL}}
\newcommand{\PGL}{\operatorname{PGL}}
\newcommand{\PSL}{\operatorname{PSL}}
\newcommand{\Irr}{\operatorname{Irr}}
\newcommand{\supp}{\operatorname{supp}}
\newcommand{\rank}{\operatorname{rank}}

\newcommand{\Ind}{\operatorname{Ind}}
\newcommand{\1}{\mathbf 1}

\theoremstyle{plain}
\newtheorem{theorem}{Theorem}[section]
\newaliascnt{proposition}{theorem}
\newtheorem{proposition}[proposition]{Proposition}
\aliascntresetthe{proposition}
\newaliascnt{lemma}{theorem}
\newtheorem{lemma}[lemma]{Lemma}
\aliascntresetthe{lemma}
\newaliascnt{corollary}{theorem}

\aliascntresetthe{corollary}
\newaliascnt{conjecture}{theorem}
\newtheorem{conjecture}[conjecture]{Conjecture}
\aliascntresetthe{conjecture}
\theoremstyle{definition}
\newaliascnt{definition}{theorem}
\newtheorem{definition}[definition]{Definition}
\aliascntresetthe{definition}
\newaliascnt{question}{theorem}
\newtheorem{question}[question]{Question}
\aliascntresetthe{question}
\theoremstyle{remark}
\newaliascnt{remark}{theorem}
\newtheorem{remark}[remark]{Remark}
\aliascntresetthe{remark}
\newaliascnt{example}{theorem}

\aliascntresetthe{example}

\title{On Shalev's Quantitative Version of Thompson's Conjecture}
\author{Ohad Sheinfeld\thanks{Einstein Institute of Mathematics, Hebrew
University. The author is partially supported by the European Research
Council (StG no.~101163794). \texttt{oshenfeld@gmail.com}.}}
\date{}

\begin{document}

\maketitle

\begin{abstract}
   The classical Thompson's conjecture asserts that in any finite nonabelian simple group $G$, there exists a conjugacy class $S$ whose square is equal to the entire group. Shalev (Annals Math., 2009) conjectured that a much stronger quantitative version holds: there exists $\epsilon > 0$ such that for any conjugacy class $S$ with 
   $|S| \geq |G|^{1-\epsilon}$, we have
   $G \setminus \{1\} \subseteq S^2$. While the conjecture turns out to  
   fail in general, we prove a close variant of the conjecture over $\PSL_n(2)$:
There exist
    \(\epsilon>0\) and \(N \in \mathbb{N}\) such that for all
    \(n\geq N\), every conjugacy class \(S\subseteq\PSL_n(2)\) with
    \(|S|\geq|\PSL_n(2)|^{1-\epsilon}\) satisfies
$
   \{x\in\PSL_n(2):\supp(x)\neq1\}\subseteq SS^{-1}.
$
On the other hand, we show that for each $\epsilon>0$, there exist conjugacy classes $S \subset \PSL_n(8)$ of size $>|\PSL_n(8))|^{1-\epsilon}$ for which $SS^{-1}$ misses even elements with support size linear in $n$. 
In addition, for every odd $n \ge 5$, we construct an inverse-closed normal set in $\PSL_n(5)$ with fixed
positive density whose square misses a nonidentity element. This answers on the negative a question of Shalev (ECM, 2021) and disproves a
conjecture of Pyber (Kourovka notebook, 2022).
\end{abstract}

\section{Introduction}
\label{sec:introduction}

Let \(G\) be a finite nonabelian simple group and let \(S\subseteq G\) be a
conjugacy class.  Shalev~\cite[Conjecture~10.3]{Shalev2009} conjectured that
there exists \(\epsilon_0>0\) such that
\begin{equation}\label{eq:intro-shalev}
|S^2|\geq
   \min\bigl\{|S|^{1+\epsilon_0},\,|G|-\delta_S\bigr\},
\end{equation}
where $\delta_S=0$ if $S=S^{-1}$ and $\delta_S=1$ is $S \neq S^{-1}$.
We call \(S\) \emph{real} when \(S=S^{-1}\). Clearly, the square of a nonreal conjugacy class  cannot contain the identity.

The minimum in~\eqref{eq:intro-shalev} separates two regimes.  In the ``growth regime'', where 
\(
   |S|\leq |G|^{1-\epsilon_1},
\) the conjecture was proved by 
Gill, Pyber, Short and Szab\'{o}~\cite[Proposition~5.2]{gill2013product}. Specifically, they proved that for every $\epsilon_1>0$ there is
\(\epsilon_2 >0\), depending only on \(\epsilon_1\), such that
\(
   |S^2|\geq |S|^{1+\epsilon_2}.
\)
We therefore focus on the complementary regime, namely the problem of establishing \(S^2\supseteq G\setminus\{1\}\) for classes satisfying \(|S|\geq|G|^{1-\epsilon}\)
for a small fixed \(\epsilon>0\).  

A famous conjecture of Thompson states that every finite nonabelian simple group
\(G\) contains a conjugacy class whose square is \(G\).  A recent breakthrough of Larsen and Tiep~\cite{LarsenTiepThompson}
proved the conjecture for all sufficiently large \(G\) using character bounds, and very recently,
Liao, Wang, and Zhang~\cite{LiaoWangZhang2026} 
solved the conjecture in full generality. 
Shalev's conjecture in this regime is a robust variant of Thompson's conjecture: it
replaces the existence of one conjugacy class whose square covers the entire group $G$ with the assertion that for every `sufficiently large' conjugacy class $S$, we have \(S^2=G\) when \(S\) is real and
\(S^2=G\setminus\{1\}\) otherwise. 

Since in the bounded rank regime there are no conjugacy classes of size at least \(|G|^{1-\epsilon}\), we study this problem in the high rank regime.
Let us introduce an important notion: The \emph{support} of an element \(g\in\GL_n(q)\) is defined as
\[\supp(g):=\min_{\lambda\in\overline{\F}_q^\times}
\rank(g-\lambda I),
\] meaning the codimension of a largest eigenspace over
the algebraic closure.

\paragraph{Shalev's conjecture fails in general.} It turns out that Shalev's conjecture does not hold in full generality, as in some cases, not only $1$ cannot be covered by elements of $S^2$, but also some other elements, which have support 1. 
Specifically, Uhlig~\cite[Lemma~2]{Uhlig1979}
(see Nielsen~\cite[Lemma~3.2]{Nielsen2026} for a short proof)
showed, that for every conjugacy class \(S\subseteq\GL_n(q)\) with an 
irreducible characteristic polynomial, \(SS^{-1}\)
does not contain some elements of support 1.  For every prime power \(q\) and every
even dimension \(n\geq4\), the conjugacy class \(S\subseteq\SL_n(q)\) can be chosen to be real, and thus, $S^2=SS^{-1}$ misses some elements with support 1. For fixed \(q\), and $n\rightarrow \infty$, we have
\[
   \frac{\log|S|}{\log|\SL_n(q)|}\longrightarrow1,
   \qquad
   S^2=SS^{-1},
\]
while \(S^2\) omits some elements of support 1.
This naturally leads to the following \emph{approximate covering} question, which we would like to explore in this paper:
\begin{question}
    Is it true that for any high rank linear group $G$ and for any conjugacy class $S$ with $|S|>|G|^{1-\epsilon}$, $S^2$ covers any element of $G$ of support at least 2? And if not, can $S^2$ be replaced by a more appropriate product to obtain such a covering result? 
\end{question}

\paragraph{A negative result for approximate covering with $S^2$.} Our first result shows that the answer to the first question is negative. 
For a monic polynomial
\(f\), write \(C(f)\) for its companion matrix. We have the following counterexample:
\begin{theorem}
\label{thm:intro-f2-obstruction}
Let \(p\in\F_2[T]\) be monic and irreducible of degree \(d\geq3\), and
suppose that its roots are not closed under inversion.  Let \(m\geq1\) and
put \(n=dm\).  Let \(S\) be the conjugacy class in \(\GL_n(2)=\PSL_n(2)\) of the
block-diagonal matrix \(C(p)^{\oplus m}\) formed from \(m\) copies of
\(C(p)\).  Then we have
\[
   \frac{\log|S|}{\log|\GL_n(2)|}
      =1-\frac1d+O(n^{-2}),
\]
but
\[
   S^2\cap\{x\in\GL_n(2):\supp(x)<m\}=\varnothing.
\] 
\end{theorem}
$\GL_n(2) = \PSL_n(2)$ is a finite simple group. Note that 
there are two natural parameter ranges here. Taking \(m=3\) and letting
\(d\to\infty\) gives \(\log|S|/\log|\GL_{3d}(2)|=1-1/d+O(d^{-2})\to1\), while
all support 2 elements remain uncovered by $S^2$. Alternatively, for any $\epsilon>0$ one can fix \(d\) sufficiently large so that \(\log|S|/\log|\GL_{3d}(2)|=1-1/d+O(d^{-2}) >1-\epsilon \) and let \(m\to\infty\), to get a conjugacy class of size at least \(|\GL_n(2)|^{1-\epsilon}\) whose square
misses every element of support below \(n/d\), which is linear in the dimension. This shows that excluding constant supports as an alternative conjecture is not the `right' solution.

\paragraph{Approximate covering with $SS^{-1}$.}
A natural replacement for $S^2$ is the difference product \(SS^{-1}\), which
always contains the identity. Our main result shows that for linear groups over $\F_2$, this variant of Shalev's conjecture indeed holds.

\begin{samepage}
\begin{theorem}
\label{thm:main-binary}
There exist \(\epsilon>0\) and
\(N\in\mathbb{N}\) such that the following holds.  Let \(n\geq N\), and let
\(S\subseteq\GL_n(2)\) be a
conjugacy class satisfying
   $|S|\geq|\GL_n(2)|^{1-\epsilon}$.
Then
\begin{equation}
   \{x\in\GL_n(2):\supp(x)\neq1\}\subseteq SS^{-1}.
   \label{eq:intro-main-binary}
\end{equation}
\end{theorem}
\end{samepage}
This repair, however, does not hold over every finite field, as is shown by the
following example. We write \(a^G\) for the conjugacy class of \(a\) in \(G\).
\begin{theorem}
\label{thm:intro-f8-obstruction}
Let \(d\geq4\) be even, and let \(m\geq3\) be such that \(7\nmid dm\).
Put \(n=dm\) and \(G=\SL_n(8)=\PSL_n(8)\).
There is a monic irreducible polynomial \(p\in\F_8[T]\) of degree \(d\)
such that the conjugacy class
\[
   S:=\bigl(C(p)^{\oplus m}\bigr)^G\subseteq G
\]
is real and satisfies $\log|S|/\log|G|
      =1-\frac1d+O(n^{-2})$. However, for every integer \(s\) with \(2\leq s<m\), there is an
element \(x_s\in G\) such
that
\(
   \supp(x_s)=s\) and \(
   x_s\notin SS^{-1}.
\)
\end{theorem}

As in~\cref{thm:intro-f2-obstruction}, taking \(m=3\) and \(d\to\infty\) gives
a real conjugacy class $S$ with \(\log|S|/\log|G|\to1\), for which \(SS^{-1}\) misses some elements of
support 2. Keeping \(d\) fixed and letting \(m\to\infty\), we obtain $S$ with $|S|>|\SL_n(8)|^{1-\epsilon}$ such that \(SS^{-1}\) misses
elements of every support \(2\leq s<n/d\), so the non-covered elements may have supports linear
in \(n\).

\medskip In the other direction, we show that the assertion of Theorem~\ref{thm:main-binary} holds over any finite field for the family of  
conjugacy classes represented by \(C(p^r)\), where \(p\) is irreducible and
\(n=r\deg p\). We call this class a cyclic-primary class.
\begin{theorem}
\label{thm:intro-primary-cyclic-difference}
There exists \(N \in \mathbb{N}\) with the following
property.  Let \(q\) be a prime power, let \(p\in\F_q[T]\) be monic
irreducible, let \(r\geq1\), and put
\(d=\deg p\), \(n=dr\geq N\), and
\(S=C(p^r)^{\PGL_n(q)}\).  Then
\[
   \{x\in\PSL_n(q):\supp(x)\neq1\}
      \subseteq SS^{-1}.
\]
Moreover, if \(r\geq2\), then
\(SS^{-1}=\PSL_n(q)\).
\end{theorem}


\paragraph{Covering with squares of normal sets.} 
A normal set in a group is a union of conjugacy classes.
In his talk at the 2021 European Congress of Mathematicians, Shalev~\cite[Question 1.5]{ShalevEMS2023} asked whether 
for any sufficiently large finite nonabelian simple group \(G\), every inverse-closed
normal subset \(S\) with a positive density satisfies $S^2=G$. In the 2022 edition of the Kourovka notebook, Pyber~\cite[Problem~20.74(a)]{Kourovka2026} conjectured that the same assertion holds under the weaker assumption $|S|>|G|/\log_2|G|$.   
Pyber's conjecture was proved for alternating groups by
Larsen and Tiep~\cite{LarsenTiepAlternating}, and for groups of Lie type
of bounded rank by Skresanov~\cite{Skresanov2025}. 

Our next result disproves Pyber's conjecture and answers on the negative Shalev's question.
\begin{theorem}\label{thm:intro-normal}
    For every odd \(n\geq5\), there exists an inverse-closed normal subset \(S\subseteq \SL_n(5)=\PSL_n(5)\) such that $|S| \geq |\SL_n|/960$, but $\operatorname{diag}(1,-I_{n-1})\notin S^2$. 
\end{theorem}
The construction extends to every fixed \(q>2\), with
a positive density depending only on \(q\).

\paragraph{Proof strategy for the main theorem.} We conclude the introduction by describing the main ideas behind our proof of Theorem~\ref{thm:main-binary}.
When asking whether \(x\in S^2\) or \(x\in SS^{-1}\), we call \(S\) the
\emph{source class} and \(x\) the \emph{target element}.

\medskip 

The proof of \cref{thm:main-binary} has three steps.  The first two steps
establish \cref{thm:intro-primary-cyclic-difference}, mainly via character bounds. The third step combines \cref{thm:intro-primary-cyclic-difference} with the Guralnick-Larsen-Tiep character
bound \cite[Theorem~1.3]{GLTClassical} and an explicit construction to get~\cref{thm:main-binary}.

We use the Frobenius criterion: for \(S=a^{\GL_n(q)}\) and
\(x\in\SL_n(q)\),
\[
   x\in SS^{-1}
   \quad\Longleftrightarrow\quad
   \sum_{\chi\in\Irr(\GL_n(q))}
      \frac{|\chi(a)|^2\chi(x^{-1})}{\chi(1)}>0.
\]

\begin{enumerate}[leftmargin=2.5em,label=\textbf{\arabic*.}]
\item \emph{Large support.}
      The linear characters give a positive main term. Our starting point is the Larsen-Tiep
      character-ratio bound (\cref{thm:larsen-tiep-character-ratio}) which gives,
      for an absolute constant \(\sigma>0\) and \(n\geq5\),
      \[
         |\chi(x)|
            \leq\chi(1)^{1-\sigma\supp(x)/n}
         \qquad\bigl(\chi\in\Irr(\GL_n(q))\bigr).
      \]
      We prove a new character bound:
      \[
         |\chi(C(p^r))|\leq2^j q^{j/2}
         \qquad\bigl(\ell(\chi)=j<n/2\bigr),
      \]
      where \(\ell(\chi)\) denotes the character level, defined in
      \cref{sec:characters-green}. There exists a constant $B>0$ such that if \(\supp(x)>B\), 
      these bounds 
      make the
      nonlinear characters' contribution in the Frobenius sum smaller than the linear characters' contribution.
      It remains to consider the case \(\supp(x)\leq B\).

\item \emph{Bounded support.}
      We use a theorem of B\"unger and Nielsen
      \cite[Theorem~4.2]{BungerNielsen1999} which shows that \(x\in SS^{-1}\)
      whenever the characteristic polynomial of \(x\) has more than one
      irreducible factor.
      In the remaining case, \(x\) has a unipotent representative in the
      projective group.  

      Put \(d=\deg p\), so that \(n=rd\).  For \(d=1\), we explicitly construct $x$ as a product of $a\in S$, $b\in S^{-1}$. For \(d\geq2\), we distuinguish between the two ways in which \(n\) can be
      large. Large $r$ allows a similar explicit factorization. When $r$
      is small, \(d\) is large, so since $x$ has bounded support the nontrivial
      Jordan blocks of \(x\) fit inside one \(d\)-dimensional block.  After
      conjugation, we write \(x=x_0\oplus I_{n-d}\), with \(x_0\in\GL_d(q)\).
      
      We prove that one can lift a factorization \(x_0=aa'^{-1}\), with \(a,a'\) conjugate to
      \(C(p)\), to a factorization \(x=bb'^{-1}\), with
      \(b,b'\) conjugate to \(C(p^r)\).

      It therefore remains to factor \(x_0\) in this way, which is easier than
      the original task since most irreducible characters vanish at \(C(p)\).
      Analyzing the nonzero character values proves positivity of the
      Frobenius sum.

\item \emph{From cyclic-primary blocks to a large source.}
We now work over \(\F_2\). Let \(a=\bigoplus_i C(p_i^{r_i})\) be a
representative of \(S\) in rational canonical form, and let \(s=\supp(x)\). For
sufficiently large \(n\) and \(s\geq n/4\), the Guralnick-Larsen-Tiep character
bound \cite[Theorem~1.3]{GLTClassical} implies \(x\in SS^{-1}\). So we may assume \(s<n/4\), so that \(x\sim x_0\oplus I\) with
\(\dim x_0\leq 2s\).

If \(a=a_0\oplus a_1\) and (after padding \(x_0\) with identity blocks) we have
\(x_0=bb'^{-1}\) for some \(b,b'\) conjugate to \(a_0\), then
\[
   (b\oplus a_1)(b'\oplus a_1)^{-1}=bb'^{-1}\oplus a_1a_1^{-1}=x_0\oplus I.
\]
This allows for a ``cancellation'' in the source: placing the same unused source block in
both factors contributes \(a_1a_1^{-1}=I\) in the target, and allows us to get rid of a block in the source. If a single source block has
dimension at least \(2s\), we can use
\cref{thm:intro-primary-cyclic-difference}  to get a padded $x_0$ on that
block, and cancel all other blocks.

If no large dimension block exists, bounded \(s\) forces repeated block types. This
produces a centralizer of size at least \(2^{cn^2}\) for some fixed \(c>0\),
which contradicts the class-size hypothesis. Thus \(s\) must exceed a fixed
absolute threshold, although it may remain small relative to \(n\). We then cancel
blocks, ensuring the remaining source conjugacy class stays sufficiently large, until
its ambient dimension \(m\) satisfies \(2s<m\leq 4s\). Since \(s\geq m/4\), the
aforementioned Guralnick-Larsen-Tiep character
bound applies in the remaining dimension.
\end{enumerate}

\section{Background and notation}
\label{sec:background}

We collect the character estimates, character-level machinery, matrix
conventions, and centralizer bounds used in the proofs. All characters are
ordinary complex characters, \(\Irr(G)\) denotes the set of irreducible
characters of \(G\), and \(\chi(1)\) is the dimension of \(\chi\). We write
\(\boxtimes\) for the external tensor product and \(A\sim B\) for similarity of
matrices over their base field, and put \(\GL_0(q)=\{1\}\). Throughout, \(S\)
denotes a source conjugacy class and \(a\in S\) a chosen representative.

\subsection{Larsen-Tiep character bounds}

We use the following consequence of the character bound of Larsen and Tiep
\cite[Theorem~5.5]{LarsenTiepUniform}.

\begin{theorem}
\label{thm:larsen-tiep-character-ratio}
There exists \(\sigma>0\) such that, for every \(n\geq5\),
every prime power \(q\), every \(x\in\SL_n(q)\), and every
\(\chi\in\Irr(\GL_n(q))\),
\begin{equation}
   \frac{|\chi(x)|}{\chi(1)}
      \leq \chi(1)^{-\sigma\supp(x)/n}.
   \label{eq:LT-ratio-GL}
\end{equation}
\end{theorem}

Larsen and Tiep originally proved this result for characters of \(\SL_n(q)\).
By Clifford theory, the restriction of \(\chi\) to \(\SL_n(q)\) is a sum of
at most \(q-1\) irreducible characters of equal degree. For \(n\geq 5\),
every nonlinear character \(\chi\) satisfies \(\chi(1)\geq q^{n-1}\)
\cite[Theorem~1.2(i),(ii)]{GLTLevels}. Therefore, the bound for \(\GL_n(q)\)
follows from the original theorem
\cite[Theorem~5.5]{LarsenTiepUniform} by halving the absolute constant.
\subsection{Characters, Green's parametrization, and Pieri's formula}
\label{sec:characters-green}

The group \(P\leq G=\GL_n(q)\) of invertible block upper-triangular matrices
with fixed block sizes is called a \emph{standard parabolic subgroup}.
Its block-diagonal subgroup \(L\) is a \emph{Levi subgroup} of \(P\).
The subgroup \(N\) of matrices in \(P\) whose diagonal blocks are identity
matrices is the \emph{unipotent radical} of \(P\). We have \(P=LN\).
The projection onto the diagonal blocks is a homomorphism \(P\to L\)
with kernel \(N\). Thus \(N\triangleleft P\) and \(P/N\cong L\).

An \(L\)-representation can be viewed as a \(P\)-representation
through the quotient map \(P\to P/N\cong L\), with \(N\) acting trivially.
This is called \emph{inflation}. The \emph{Harish-Chandra induction} of an
\(L\)-representation is the \(G\)-representation obtained by first inflating
it to \(P\) and then inducing it to \(G\).

Its adjoint, \emph{Harish-Chandra restriction}, sends a representation
\(W\) of \(G\) to a representation of \(L\) by restricting \(W\) to \(P\)
and then taking the fixed subspace
\[
   W^N=\{w\in W:zw=w\text{ for every }z\in N\}.
\]
Since \(N\) is normal in \(P\), this subspace is \(P\)-stable, and the action
factors through \(P/N\cong L\). Notice that
ordinary restriction to \(L\) keeps the whole space \(W\), whereas
Harish-Chandra restriction keeps only \(W^N\), which may be zero. The averaging
operator \(A_N:W\to W\), defined by \(A_N(w)=|N|^{-1}\sum_{z\in N}zw\), is a
projection onto \(W^N\). Its image lies in \(W^N\), since multiplication by any
element of \(N\) permutes the summands, and it fixes every vector of \(W^N\).
Taking traces gives
\begin{equation}
   \label{eq:fixed-space-average}
   \dim W^N=\operatorname{tr}(A_N)
      =\frac{1}{|N|}\sum_{z\in N}\chi(z),
\end{equation}
where \(\chi\) is the character of the \(G\)-representation \(W\).

An irreducible character is \emph{cuspidal}
if it is not a component of a Harish-Chandra induction from any proper
Levi subgroup.  Equivalently, all of its Harish-Chandra restrictions to
proper Levi subgroups are zero.

In Green's parametrization, a character \(\chi\in\Irr(\GL_n(q))\) is labeled by
a finitely supported family of partitions \((\lambda_\varphi)_\varphi\), indexed
by cuspidal irreducible characters \(\varphi\in\Irr(\GL_e(q))\) for varying
\(e\), such that
\[
   \sum_\varphi \deg(\varphi)\,|\lambda_\varphi|=n,
   \qquad \deg(\varphi):=e.
\]
See \cite[Proposition~9.4]{Zelevinsky1981}.
We refer to the indices \(\varphi\) as \emph{cuspidal colors}.

We use the following form of Harish-Chandra's uniqueness theorem.

\begin{theorem}[Uniqueness of cuspidal support
  {\cite[Theorem~4.1]{Prasad2014Notes}}]
\label{thm:cuspidal-support}
Let \(q\) be a prime power and let \(\chi\) be an irreducible complex
character of \(\GL_n(q)\).  There exist positive integers
\(e_1,\ldots,e_r\) with \(e_1+\cdots+e_r=n\) and cuspidal
characters \(\varphi_i\in\Irr(\GL_{e_i}(q))\) such that \(\chi\) is a
component of Harish-Chandra induction of
\(\varphi_1\boxtimes\cdots\boxtimes\varphi_r\) from the block-diagonal Levi
subgroup \(\prod_{i=1}^r\GL_{e_i}(q)\).
The pairs \((e_1,\varphi_1),\ldots,(e_r,\varphi_r)\), including their
multiplicities, are uniquely determined by \(\chi\) up to permutation.
\end{theorem}

The multiset of characters \(\varphi_1,\ldots,\varphi_r\) in
\cref{thm:cuspidal-support} is called the \emph{cuspidal support} of \(\chi\).
In Green's parametrization, this multiset contains each \(\varphi\) exactly
\(|\lambda_\varphi|\) times.

A character is \emph{primary} if exactly one partition \(\lambda_\varphi\)
in its Green parameter is nonempty.  In this case we write
\(\chi=\chi^{\varphi,\lambda}\), where \(\lambda=\lambda_\varphi\) and all other
partitions are empty.  Writing \(e=\deg(\varphi)\), we have
\[
   n=e|\lambda|,
   \qquad \lambda\vdash n/e.
\]
For a primary character we call \(e\) the \emph{cuspidal degree} of \(\chi\).
Its cuspidal support consists of \(|\lambda|=n/e\) copies of \(\varphi\).
A primary character is itself cuspidal precisely when
 \(e=n\), and in this case
\(\chi^{\varphi,(1)}=\varphi\).

Let \(\1\) denote the trivial color, namely the trivial character of
\(\GL_1(q)\).  The \emph{unipotent characters} of \(\GL_n(q)\) are
\(\chi_\lambda:=\chi^{\1,\lambda}\), with \(\lambda\vdash n\).
A primary character is a \emph{primary hook} if its partition is a hook
\(\lambda=(m-j,1^j)\), where \(m=n/e\) and \(0\leq j<m\).

The \emph{true level} \(\ell^*(\chi)\) of
\(\chi\in\Irr(\GL_n(q))\) is the least \(j\geq0\) for which \(\chi\)
appears as an irreducible component of the permutation representation arising
from the action of \(\GL_n(q)\) on \((\F_q^n)^j\).
The \emph{level} \(\ell(\chi)\) is the least true level among the twists of
\(\chi\) by linear characters, see~\cite{GLTLevels}.

Set \((\lambda_\varphi)_1=0\) when \(\lambda_\varphi\) is empty.
The level is read directly from the first rows: by
\cite[Theorem~3.13(i),(iii)]{GLTLevels},
\[
   \ell^*(\chi)=n-(\lambda_{\1})_1,
   \qquad
   \ell(\chi)
      =n-\max_{\alpha\in\Irr(\GL_1(q))}(\lambda_\alpha)_1.
\]
In particular, a primary character of cuspidal degree \(e>1\) has
\(\ell^*(\chi)=\ell(\chi)=n\).

We make use of the following dimension bounds.

\begin{theorem}[Guralnick-Larsen-Tiep
  {\cite[Theorem~1.2(i),(ii)]{GLTLevels}}]
\label{thm:level-dimension-bounds}
Let \(n\geq3\), and let \(\chi\in\Irr(\GL_n(q))\) have level
\(j=\ell(\chi)\).  Then
\[
   q^{j(n-j)}\leq\chi(1)\leq q^{nj}.
\]
If \(j\geq n/2\), then the following lower bound holds
\[
   \chi(1)>\frac{9}{16}(q-1)q^{n^2/4-1}>q^{n^2/4-2}.
\]

\end{theorem}

For small true level, Green's parametrization also gives a convenient
indexing.  If \(\ell^*(\chi)=j\) and \(j\leq n/2\), then
\(\lambda_{\1}=(n-j,\mu)\), and deleting its first row leaves a Green parameter
of weighted rank \(j\), namely
\[
   |\mu|+\sum_{\varphi\neq\1}\deg(\varphi)\,|\lambda_\varphi|=j.
\]
The converse is also true: when $j \le n/2$, adjoining a first row of length
\(n-j\) to the trivial-color partition of any Green parameter of weighted rank
\(j\) gives a valid partition.  Thus for $j \le n/2$, the true-level-\(j\)
characters of \(\GL_n(q)\) are naturally parametrized by \(\Irr(\GL_j(q))\)
\cite[Theorem~3.9(i)]{GLTLevels}.  This is the form used in the inverse Pieri
formula below.

For \(k,m\geq0\), let \(P_{k,m}\) denote the standard parabolic subgroup of
\(\GL_{k+m}(q)\) with Levi factor \(\GL_k(q)\times\GL_m(q)\).  We first
recall the usual Pieri rule.  Suppose that
\(\theta\in\Irr(\GL_k(q))\) has trivial-color partition \(\rho\).  Then
\begin{equation}
   \Ind_{P_{k,m}}^{\GL_{k+m}(q)}(\theta\boxtimes\1)
      =\sum_{\psi}\psi,
   \label{eq:ordinary-pieri}
\end{equation}
where \(\psi\) runs over the characters obtained by adding a horizontal
\(m\)-strip to \(\rho\) and leaving all other cuspidal components unchanged (for example~\cite[Lemma~2.4]{StableCharactersGL}).
Here a horizontal \(m\)-strip consists of \(m\) boxes, no two in the same
column, and every character in the sum occurs with multiplicity one.  We need the inverse formula.

\begin{lemma}
\label{lem:stable-pieri-inversion}
Suppose that \(n\geq2j\) and \(\chi\in\Irr(\GL_n(q))\) has true level \(j\).
Write \(\lambda_{\1}=(n-j,\mu)\), and put \(k=j-|\mu|\).  Thus \(k\) is
the weighted rank carried by the nontrivial cuspidal components.  For every
\(\nu\subseteq\mu\) such that \(\mu/\nu\) is a vertical strip - that is, it
has at most one box in each row - let
\(\theta_\nu\in\Irr(\GL_{k+|\nu|}(q))\) have trivial-color partition
\(\nu\) and the same nontrivial cuspidal components as \(\chi\).  Then
\begin{equation}
   \chi
   =\sum_{\substack{\nu\subseteq\mu\\
             \mu/\nu\text{ a vertical strip}}}
      (-1)^{|\mu|-|\nu|}
      \Ind_{P_{k+|\nu|,\,n-k-|\nu|}}^{\GL_n(q)}
         (\theta_\nu\boxtimes\1).
   \label{eq:exact-pieri-inversion}
\end{equation}
And the sum has at most \(2^j\) terms.
\end{lemma}
The lemma follows from inverse Pieri for the symmetric group, because the two
Pieri rules have the same effect on the trivial-color partition.  We make
this comparison explicit in the following proof.
\begin{proof}
Fix the nontrivial cuspidal components and set \(N=n-k\).  The
trivial-color partition of \(\chi\) is
\[
   (n-j,\mu)=(N-|\mu|,\mu).
\]
Let \(\operatorname{Sp}_\lambda\) denote the irreducible character of
\(S_{|\lambda|}\) indexed by \(\lambda\).  The ordinary Pieri rule for \(S_N\)
is
\[
   \Ind_{S_{|\rho|}\times S_{N-|\rho|}}^{S_N}
      (\operatorname{Sp}_\rho\boxtimes\1)
   =\sum_{\substack{\lambda\vdash N\\
              \lambda/\rho\text{ a horizontal strip}}}
      \operatorname{Sp}_\lambda.
\]
Since \(N-|\mu|=n-j\geq j\geq\mu_1\), the inverse Pieri formula gives
\[
   \operatorname{Sp}_{(N-|\mu|,\mu)}
      =\sum_{\substack{\nu\subseteq\mu\\
              \mu/\nu\text{ a vertical strip}}}
         (-1)^{|\mu|-|\nu|}
         \Ind_{S_{|\nu|}\times S_{N-|\nu|}}^{S_N}
            (\operatorname{Sp}_\nu\boxtimes\1),
\]
see \cite[Proposition~2.1, equations~(2.1)--(2.2)]{AssafSpeyer2020}.

The inverse Pieri formula is the matrix inverse of the ordinary Pieri rule.
After the nontrivial cuspidal components are fixed,
\eqref{eq:ordinary-pieri} and the symmetric-group Pieri rule have the same
matrix: an entry is one precisely when the larger partition is obtained from
the smaller by adding a horizontal strip, and is zero otherwise.  Hence the
signed vertical-strip matrix above also inverts \eqref{eq:ordinary-pieri}.
Restoring the fixed components adds \(k\) to the
rank of each inner character: the term \(\operatorname{Sp}_\nu\) becomes
\(\theta_\nu\in\Irr(\GL_{k+|\nu|}(q))\), and the trivial factor has rank
\(N-|\nu|=n-k-|\nu|\).  Thus the symmetric-group identity becomes exactly
\eqref{eq:exact-pieri-inversion}.

Finally, a vertical strip is determined by the rows from which its boxes are
removed, so the number of terms is at most \(2^{|\mu|}\leq2^j\).
\end{proof}

\subsection{Companion blocks and reciprocal polynomials}

For a monic polynomial \(f\), let \(C(f)\) denote its companion matrix.  Let
\(p\in\F_q[T]\) be monic, irreducible, of degree \(d\).
Its monic reciprocal polynomial is
\[
   p^\vee(T):=p(0)^{-1}T^d p(T^{-1}).
\]
The roots of \(p^\vee\) are the inverses of the roots of \(p\).  Consequently,
\[
   C(p)^{-1}\sim C(p^\vee).
\]
More generally, inversion sends the cyclic-primary class of \(C(p^r)\) to
that of \(C((p^\vee)^r)\).

We call a matrix \emph{primary of degree \(e\)} if its characteristic
polynomial is a power of one irreducible polynomial of degree \(e\).  It is
\emph{nonprimary} otherwise.  A matrix is \emph{cyclic} if its characteristic
and minimal polynomials coincide.

\subsection{Unipotent classes and dominance order}
\label{sec:unipotent-classes}

Unipotent conjugacy classes in \(\GL_n(q)\) are indexed by
partitions \(\tau=(\tau_1\geq\cdots\geq\tau_h>0)\) of \(n\):
a matrix of type \(\tau\) has Jordan form
\[
   J_{\tau_1}(1)\oplus\cdots\oplus J_{\tau_h}(1),
   \qquad \tau_1+\cdots+\tau_h=n,
\]
where \(J_m(1)\in\GL_m(q)\) has ones on the diagonal and superdiagonal and zeros
elsewhere.  The conjugate partition \(\tau'\) has \(j\)-th part
\(\tau'_j=|\{a:\tau_a\geq j\}|\), the number of Jordan blocks of size
at least \(j\).  Thus \(\tau'_1=h\) is the number of Jordan blocks,
and the support of a matrix of type \(\tau\) is \(n-h\).

If \(u=I+A\) has type \(\tau\), each Jordan block of size \(\tau_a\)
contributes \(\min\{i,\tau_a\}\) to \(\dim\ker A^i\).  Summing over
the blocks gives
\begin{equation}
   \dim\ker A^i
      =\sum_{a=1}^h\min\{i,\tau_a\}
      =\sum_{j=1}^i\tau'_j,
   \qquad i\geq1,
   \label{eq:unipotent-kernel-dimensions}
\end{equation}
where missing parts are treated as zero.  The last sum counts the first
\(i\) positions present in each Jordan chain.

The \emph{dominance order} on partitions of \(n\) can be read from either
the parts or the conjugate parts, with the inequality reversed:
\begin{equation}
   \begin{aligned}
   \mu\leq\tau
   &\quad\Longleftrightarrow\quad
   \sum_{j=1}^i\mu_j\leq\sum_{j=1}^i\tau_j
   &&\text{for every }i\geq1,\\
   &\quad\Longleftrightarrow\quad
   \sum_{j=1}^i\mu'_j\geq\sum_{j=1}^i\tau'_j
   &&\text{for every }i\geq1.
   \end{aligned}
   \label{eq:dominance-conjugate-partitions}
\end{equation}
By \eqref{eq:unipotent-kernel-dimensions}, moving downward in dominance
order cannot decrease the kernel dimension of any power of the nilpotent part.

\subsection{Centralizers and the natural class-size scale}

For any finite group \(G\), we have the standard
character centralizer bound: for \(g\in G\) and \(\chi\in\Irr(G)\),
\begin{equation*}
   1 = 1/|G| \cdot \sum_h|\chi(h)|^2
      \ge\frac{|\chi(g)|^2 \cdot |g^G|}{|G|}
      =|\chi(g)^2|/|C_G(g)|.
\end{equation*}
So 
\begin{equation}
    |C_G(g)|\ge |\chi(g)|^2.
    \label{eq:character-centralizer-bound}
\end{equation}
For \(g\in\GL_n(q)\), the size of its conjugacy class is
\[
   |g^{\GL_n(q)}|
   =\frac{|\GL_n(q)|}{|C_{\GL_n(q)}(g)|}.
\]
For the cyclic-primary block, put \(n=dr\).  Its centralizer is the unit group
of its endomorphism ring:
\[
   C_{\GL_n(q)}\bigl(C(p^r)\bigr)
   \cong
   \bigl(\F_q[T]/(p^r)\bigr)^\times,
\]
and hence
\[
   \left|C_{\GL_n(q)}\bigl(C(p^r)\bigr)\right|
   =q^{d(r-1)}(q^d-1).
\]
Consequently,
\[
   \log_q\left|C(p^r)^{\GL_n(q)}\right|
      =n^2-n+O_q(1).
\]

For the repeated semisimple block
\[
   a=C(p)^{\oplus m}\in\GL_{dm}(q),
   \qquad S=a^{\GL_{dm}(q)},
\]
the \(p\)-action identifies the natural space with an \(m\)-dimensional
space over \(\F_{q^d}\).  Therefore
\begin{equation}
   C_{\GL_{dm}(q)}(a)\cong\GL_m(q^d),
   \label{eq:repeated-block-centralizer}
\end{equation}
so, writing \(n=dm\),
\begin{align*}
   \log_q |S|
      &=n^2-dm^2+O_q(1),\\
   \frac{\log |S|}{\log|\GL_n(q)|}
      &=1-\frac1d+O_q(n^{-2}).
\end{align*}

For a general \(D\in\GL_m(q)\), let \(b_{p,r}\) be the multiplicity of
\(C(p^r)\) in its elementary-divisor decomposition, where \(p\) is monic
irreducible with \(p(0)\neq0\) and \(r\geq1\).  By the standard
centralizer-order formula \cite[proof of Theorem~6.4]{FulmanGuralnick2012},
the proportion of matrices commuting with \(D\) that are invertible is
\begin{equation}
   \frac{|C_{\GL_m(q)}(D)|}{|C_{\operatorname{Mat}_m(q)}(D)|}
      =\prod_{p,r}\prod_{j=1}^{b_{p,r}}(1-q^{-j\deg p}).
   \label{eq:general-centralizer-proportion}
\end{equation}

\section{Character estimates}
\label{sec:character-bound}

In this section we prove character bounds for cyclic-primary
elements, and an exact character-ratio identity for unipotent elements of small
support.

\subsection{A level bound for cyclic-primary sources}

Let \(p\in\F_q[T]\) be monic irreducible of degree \(d\), let \(r\geq1\),
and put \(n=rd\) and \(a=C(p^r)\in\GL_n(q)\).

We first use the following consequence of the Pieri inversion formula.

\begin{lemma}
\label{lem:pieri-cyclic-evaluation}
Suppose \(n\geq2j\) and \(\chi\in\Irr(\GL_n(q))\) has true level \(j\).  There
are integers \(0\leq k_i\leq j\), characters \(\theta_i\in\Irr(\GL_{k_i}(q))\),
and signs \(\varepsilon_i\in\{1,-1\}\), with at most \(2^j\) indices \(i\), such
that
\begin{equation}
   \chi(C(p^r))
   =\sum_i\varepsilon_i \cdot 
      \begin{cases}
         \theta_i(C(p^{k_i/d})),&d\mid k_i,\\
         0,&d\nmid k_i.
      \end{cases}
   \label{eq:pieri-cyclic-evaluation}
\end{equation}
For \(k_i=0\), the displayed character value is interpreted as \(1\).
\end{lemma}

\begin{proof}
Apply \cref{lem:stable-pieri-inversion}.  A term indexed by a vertical
\(t\)-strip has sign \((-1)^t\), rank \(k=j-t\), and the form
\[
   \Ind_{P_{k,n-k}}^{\GL_n(q)}(\theta\boxtimes\1),
   \qquad \theta\in\Irr(\GL_k(q)).
\]
There are at most \(2^j\) such terms.

We now evaluate one of them at \(a=C(p^r)\).  The coset space
\(\GL_n(q)/P_{k,n-k}\) is the set of \(k\)-dimensional subspaces of
\(\F_q^n\).  The induced-character formula therefore gives
\begin{equation}
   \Ind_{P_{k,n-k}}^{\GL_n(q)}(\theta\boxtimes\1)(a)
      =\sum_{\substack{W\leq\F_q^n\\
                 \dim W=k,\ aW=W}}
         \theta(a|_W).
   \label{eq:induced-character-invariant-subspaces}
\end{equation}

Since \(a=C(p^r)\) is cyclic and \(p\)-primary, Brickman and Fillmore
\cite[Lemma~2]{BrickmanFillmore1967} show that its invariant subspaces form
the chain
\[
   \ker\bigl(p(a)^h\bigr),\qquad 0\leq h\leq r.
\]
The subspace \(\ker\bigl(p(a)^h\bigr)\) has dimension \(dh\), and the
restriction of \(a\) to it is similar to
\(C(p^h)\).  Thus \eqref{eq:induced-character-invariant-subspaces} equals
\[
   \begin{cases}
      \theta(C(p^{k/d})),&d\mid k,\\
      0,&d\nmid k.
   \end{cases}
\]
For \(k=0\), the unique zero-dimensional subspace contributes \(1\).
Indexing the Pieri terms by \(i\), with \(k_i=j-t\) and
\(\varepsilon_i=(-1)^t\), gives \eqref{eq:pieri-cyclic-evaluation}.
\end{proof}

Combining this evaluation with the centralizer bound gives the level bound.

\begin{proposition}
\label{prop:primary-small-level}
Let \(\chi\in\Irr(\GL_n(q))\) have level \(\ell(\chi)=j<n/2\).  Then
\begin{equation}
   |\chi(C(p^r))|\leq 2^j q^{j/2}.
   \label{eq:primary-character-bound}
\end{equation}
\end{proposition}

For low-level characters, this improves on the general centralizer bound
\(|\chi(a)|<q^{n/2}\).

\begin{proof}
Twist \(\chi\) so that its true level is \(j\).  This does not change
\(|\chi(a)|\).  By \cref{lem:pieri-cyclic-evaluation} and the standard
centralizer bound \eqref{eq:character-centralizer-bound}, every nonzero
summand with \(k_i>0\) satisfies
\[
   \begin{aligned}
   \left|\theta_i(C(p^{k_i/d}))\right|^2
      &\leq
      \left|C_{\GL_{k_i}(q)}\bigl(C(p^{k_i/d})\bigr)\right| \\
      &<q^{k_i}\leq q^j.
   \end{aligned}
\]
The term with \(k_i=0\) equals \(1\).
Thus every summand has absolute value at most \(q^{j/2}\).  Since there are
at most \(2^j\) summands, the triangle inequality proves
\eqref{eq:primary-character-bound}.
\end{proof}

\subsection{An exact ratio below the cuspidal degree}

Recall that a primary character has a single nonempty partition in
Green's parametrization: it is written \(\chi^{\varphi,\lambda}\), with
\(\varphi\) cuspidal for \(\GL_e(q)\) and \(\lambda\vdash n/e\).
The integer \(e\), the degree of the cuspidal color \(\varphi\), is its
cuspidal degree.

\begin{lemma}
\label{lem:primary-exact-ratio}
Let \(\chi\in\Irr(\GL_n(q))\) be primary of cuspidal degree \(e\), and
let \(u\in\GL_n(q)\) be unipotent of support \(s<e\).  Then
\begin{equation}
   \frac{\chi(u)}{\chi(1)}
      =\frac{(-1)^s}{\displaystyle\prod_{i=n-s}^{n-1}(q^i-1)}.
   \label{eq:primary-exact-ratio}
\end{equation}
For \(s=0\), the empty product is interpreted as \(1\).
\end{lemma}

The proof extends the averaging argument in
\cite[proof of Proposition~4.4(c), equations~(4.9)--(4.10)]{LarsenTiepThompson}.

\begin{proof}
Choose a cuspidal character \(\rho\in\Irr(\GL_n(q))\), and set
\[
   f=\frac{\chi}{\chi(1)}-\frac{\rho}{\rho(1)}.
\]
We prove that \(f(u)=0\) by induction on the Jordan type of \(u\),
ordered by dominance (see \cref{sec:unipotent-classes}).
The minimum type \((1^n)\) represents the
identity, and \(f(1)=0\).
For the induction step, assume \(s>0\), let \(\tau\vdash n\) be
the Jordan type of \(u\), and suppose that \(f\) vanishes on every
type \(\mu<\tau\).  Write \(\tau'=(b_1,\ldots,b_k)\) for the conjugate partition
of \(\tau\), so that \(b_1=n-s\).
Choose a flag
\[
   0=V_0<V_1<\cdots<V_k=\F_q^n,
   \qquad \dim(V_i/V_{i-1})=b_i.
\]
Let \(P_\tau=L_\tau N_\tau\) be its stabilizer.  Here
\(L_\tau\cong\prod_i\GL_{b_i}(q)\), and the unipotent radical
\(N_\tau\) consists of the elements acting trivially on every
quotient \(V_i/V_{i-1}\).

The cuspidal support of \(\chi\) consists of \(n/e\) copies of its cuspidal
color \(\varphi\).  If its Harish-Chandra restriction to \(L_\tau\) were
nonzero, choose an irreducible component
\({\eta_1\boxtimes\cdots\boxtimes\eta_k}\), with
\(\eta_i\in\Irr(\GL_{b_i}(q))\).  By reciprocity, \(\chi\) occurs in
Harish-Chandra induction from this component.  Each \(\eta_i\) is itself a
component of induction from its cuspidal support.  By transitivity, the combined
supports therefore give cuspidal inducing data for \(\chi\).  By
\cref{thm:cuspidal-support}, the cuspidal supports of all the \(\eta_i\),
combined, consist of exactly those \(n/e\) copies of \(\varphi\). Thus each
\(\eta_i\) has some number \(m_i\) of copies of \(\varphi\) in its cuspidal
support, giving \(b_i=e m_i\). But \(e\mid n\) and \(0<s<e\), so
\(e\nmid n-s=b_1\). Thus the \(N_\tau\)-fixed space of \(\chi\) is zero.  The
same holds for \(\rho\) by cuspidality.  Applying \eqref{eq:fixed-space-average}
to \(\chi\) and \(\rho\) gives
\[
   \sum_{z\in N_\tau}f(z)=0.
\]

For \(z=I+A\in N_\tau\), we have \(AV_i\subseteq V_{i-1}\), hence
\(V_i\subseteq\ker A^i\).  If \(\mu\) is the Jordan type of \(z\),
then \eqref{eq:unipotent-kernel-dimensions} gives
\[
   \sum_{j=1}^i\mu'_j
      =\dim\ker A^i
      \geq\dim V_i
      =\sum_{j=1}^i\tau'_j
      \qquad(1\leq i\leq k).
\]
For \(i\geq k\), both sums equal \(n\), since \(A^k=0\). Thus
\eqref{eq:dominance-conjugate-partitions} gives \(\mu\leq\tau\). Moreover, a
conjugate of \(u\) lies in \(N_\tau\): place the \(i\)-th vector of each Jordan
chain in the \(i\)-th flag block, and let \(A\) send it to its predecessor. The
vanishing sum $\sum_{z\in N_\tau}f(z)=0$ consists of elements with type strictly
below \(\tau\), which give zero by induction, and elements conjugate to $u$.
This forces \(f(u)=0\).

It follows that \(\chi(u)/\chi(1)=\rho(u)/\rho(1)\).
Green's cuspidal character formula, in the form
\cite[Theorem~2]{Prasad2000}, gives
\[
   \rho(u)=(-1)^s\prod_{i=1}^{n-s-1}(q^i-1),
   \qquad
   \rho(1)=\prod_{i=1}^{n-1}(q^i-1).
\]
Dividing proves \eqref{eq:primary-exact-ratio}.
\end{proof}

The ratio in \eqref{eq:primary-exact-ratio} is independent of the cuspidal
degree \(e>s\). In \cref{prop:general-q-irreducible-difference}, this lets us control
all such degrees together.

\section{Cyclic-primary difference products over arbitrary fields}
\label{sec:general-q-difference}

In this section we prove \cref{thm:intro-primary-cyclic-difference}. We first fix notation
and restate the theorem.

Fix a prime power \(q\), let \(p\in\F_q[T]\) be monic irreducible with
\(p(0)\neq0\), and put \(d:=\deg p\geq1\).  For \(r\geq1\), put
\[
   S_r:=C(p^r)^{\GL_{dr}(q)}.
\]
When \(\PGL_{dr}(q)\) is the conjugating group, we view \(C(p^r)\) as
its image in \(\PGL_{dr}(q)\).
For \(a,b\in S_r\), the element \(a^{-1}b\) is conjugate to \(ba^{-1}\).
Thus
\[
   S_t^{-1}S_t=S_tS_t^{-1},
\]
allowing us to use the two set products interchangeably.

\begin{theorem}[Restatement of
Theorem~\ref{thm:intro-primary-cyclic-difference}]
\label{thm:primary-cyclic-difference}
There exists \(N \in \mathbb{N}\) such that the following
holds.  If \(n=dr\geq
N\), then
\begin{equation}
   \bigl\{x\in\PSL_n(q):\supp(x)\neq1\bigr\}
      \subseteq C(p^r)^{\PGL_n(q)}
         \bigl(C(p^r)^{\PGL_n(q)}\bigr)^{-1}.
   \label{eq:primary-difference-cover}
\end{equation}
If \(r\geq2\), then the remaining support-one elements are covered as
well, and hence
\begin{equation}
   C(p^r)^{\PGL_n(q)}
      \bigl(C(p^r)^{\PGL_n(q)}\bigr)^{-1}=\PSL_n(q).
   \label{eq:primary-difference-all}
\end{equation}
\end{theorem}

We first use character bounds to cover targets of large support, then
construct the factorizations needed for targets of bounded support.
We prove \cref{thm:primary-cyclic-difference} at the end of the section.

\begin{proposition}
\label{prop:constant-support-reduction}
There exists \(B>0\) such that, if \(n=dr\geq8\) and
\(x\in\SL_n(q)\) satisfies \(\supp(x)\geq B\), then
\[
   x\in S_rS_r^{-1}.
\]
\end{proposition}

\begin{proof}
Put \(a=C(p^r)\), and consider the Frobenius sum
\[
   \sum_{\chi\in\Irr(\GL_n(q))}
      \frac{|\chi(a)|^2\chi(x^{-1})}{\chi(1)}.
\]
The \(q-1\) linear characters contribute \(q-1\), so it suffices to show that
the nonlinear contribution has absolute value less than \(1\).
Let \(\sigma\) be as in \cref{thm:larsen-tiep-character-ratio}.  We split the
nonlinear characters at level \(n/2\), using
\cref{prop:primary-small-level} below this threshold and a centralizer bound
at or above it.

\medskip
\noindent\emph{Low levels: \(1\leq j<n/2\).}

Characters of true level \(j\) are parametrized by \(\Irr(\GL_j(q))\).
Allowing determinant twists and using
\cite[Proposition~3.5(2)]{FulmanGuralnick2012}, we get
\[
   \#\{\chi\in\Irr(\GL_n(q)):\ell(\chi)=j\}
      \leq(q-1)|\Irr(\GL_j(q))|
      \leq q^{j+1}.
\]
Moreover,
\cref{thm:level-dimension-bounds} gives
\(\chi(1)\geq q^{j(n-j)}>q^{nj/2}\).  Thus
\cref{thm:larsen-tiep-character-ratio}, together with
\cref{prop:primary-small-level}, bounds the total low-level contribution by
\begin{equation}
   \sum_{1\leq j<n/2}
      q^{j+1}\,4^j q^j\,q^{-\sigma\supp(x)j/2}
   \leq
   \sum_{j\geq1}
      \left(4q^{3-\sigma\supp(x)/2}\right)^j.
   \label{eq:low-level-error}
\end{equation}
For a sufficiently large absolute \(B\), this geometric series is uniformly
smaller than \(1/2\) whenever \(\supp(x)\geq B\).

\medskip
\noindent\emph{High levels: \(j\geq n/2\).}

The centralizer bound \eqref{eq:character-centralizer-bound} gives
\[
   |\chi(a)|^2\leq |C_{\GL_n(q)}(a)|<q^n,
\]
whereas \cref{thm:level-dimension-bounds} gives
\(\chi(1)>q^{n^2/4-2}\), and there are at most \(q^n\) irreducible
characters.  For \(n\geq8\), this gives
\(\chi(1)>q^{n^2/5}\).  Hence
\cref{thm:larsen-tiep-character-ratio} and these
estimates bound their total contribution by
\[
   q^{2n-\sigma\supp(x)n/5}.
\]
For sufficiently large \(B\), this is smaller than \(1/2\).
Thus the nonlinear contribution cannot cancel the linear one, and the
Frobenius criterion proves the assertion.
\end{proof}

It remains to treat targets of bounded support.

We need the hypothesis \(d\geq2\) later, so we first treat \(d=1\) separately.

For \(q\geq4\), the degree-one case follows from Lev's stronger theorem that
the product of two prescribed cyclic classes, one of them triangularizable,
covers every nonscalar matrix of the correct determinant
\cite[Theorem~2]{Lev1994}.  Since our result is uniform in \(q\), we record a
short argument that also covers \(q=2,3\).

\begin{lemma}
\label{lem:regular-unipotent-difference}
Let \(S\) be the conjugacy class of \(J_n(1)\) in \(\GL_n(q)\), let
\(u\in\GL_n(q)\) be unipotent, and put
\(s=\rank(u-I)\).  If \(n\geq3s\), then
\[
   u\in S^{-1}S.
\]
\end{lemma}

\begin{proof}
If \(s=0\), then \(u=I=J_n(1)^{-1}J_n(1)\), so the result is immediate.
Put \(M=u-I\) and assume \(s>0\).  Choose a
Jordan basis for \(M\), and let the nontrivial block sizes be
\(b_1,\ldots,b_k\geq2\).  Write the basis of the \(i\)-th such block as
\(v_{i,1},\ldots,v_{i,b_i}\), where
\[
   Mv_{i,j+1}=v_{i,j}\qquad(1\leq j<b_i).
\]
Then
\[
   s=\sum_{i=1}^k(b_i-1),
   \qquad
   \sum_{i=1}^k b_i=s+k\leq2s.
\]
Thus at least \(n-2s\geq s\) Jordan blocks have size one.  Choose distinct
vectors \(w_{i,j}\) from these blocks, one for each pair \((i,j)\) with
\(1\leq j<b_i\).  Reorder the basis by listing, for each \(i\),
\[
   v_{i,1},w_{i,1},v_{i,2},w_{i,2},\ldots,
   w_{i,b_i-1},v_{i,b_i},
\]
followed by the unused vectors from the one-dimensional blocks.  In this
basis, \(M\) is strictly upper triangular and has zero first superdiagonal:
each nonzero entry linking \(v_{i,j+1}\) to \(v_{i,j}\) now skips the
intervening vector \(w_{i,j}\).

Let \(J=J_n(1)=I+L\) in the same basis, so
\(L\) has ones on the first superdiagonal.  We have
\[
   Ju-I=L+M+LM,
\]
while the product \(LM\) of two strictly upper triangular matrices has zero
first superdiagonal.  Hence the first superdiagonal of \(Ju-I\) comes entirely
from \(L\) and consists of ones.  Thus
\[
   \bigl((Ju-I)^{n-1}\bigr)_{1n}=1.
\]
Since \(Ju\) is upper unitriangular, this shows that \(Ju-I\) has nilpotency
index \(n\).  Hence \(Ju\) is conjugate to \(J\).  It
follows that
\[
   u=J^{-1}(Ju)\in S^{-1}S.
\]

\end{proof}

\medskip
For \(d\geq2\), we use two ingredients: lifting factorizations and a
covering result for targets of fixed support using a single irreducible
block. We begin with the lifting construction.

\begin{proposition}
\label{prop:add-primary-layer}
Assume \(d\geq2\).  For every \(t\geq1\), every
\(x\in S_t^{-1}S_t\), and every
\(Z\in\operatorname{Mat}_{dt\times d}(\F_q)\), one has
\[
   \begin{pmatrix}x&Z\\0&I_d\end{pmatrix}
      \in S_{t+1}^{-1}S_{t+1}.
\]
\end{proposition}

\begin{proof}
Fix a root \(\beta\in\F_{q^d}\) of \(p\), and choose \(a_0\in S_1\).
For \(A\in S_t\), let \(\ell_A\in(\F_{q^d})^{dt}\) be a nonzero
\(\beta\)-eigenvector of \(A^{\mathsf T}\).  Let
\(v\in(\F_{q^d})^d\) be a nonzero \(\beta\)-eigenvector of \(a_0\).
We first show that, for every
\(W\in\operatorname{Mat}_{dt\times d}(\F_q)\),
\begin{equation}
   \begin{pmatrix}A&W\\0&a_0\end{pmatrix}\in S_{t+1}
   \quad\Longleftrightarrow\quad
   \ell_A^{\mathsf T}Wv\neq0.
   \label{eq:one-layer-cyclicity-test}
\end{equation}

Since \(A\) is conjugate to \(C(p^t)\), over \(\F_{q^d}\) it has exactly one
Jordan block of size \(t\) for the eigenvalue \(\beta\).  Hence \(A-\beta I\)
has one-dimensional left kernel, spanned by \(\ell_A^{\mathsf T}\), and
therefore
\[
   \operatorname{im}(A-\beta I)=\ell_A^\perp.
\]
A \(\beta\)-eigenvector of the displayed block matrix has the form
\((y,cv)\), where
\[
   (A-\beta I)y=-cWv.
\]
For \(c=0\) this gives the usual one-dimensional eigenspace of \(A\).  An
eigenvector with \(c\neq0\) exists exactly when \(\ell_A^{\mathsf T}Wv=0\). Thus
the block matrix has a one-dimensional \(\beta\)-eigenspace exactly when
\(\ell_A^{\mathsf T}Wv\neq0\).  The block matrix has characteristic polynomial
\(p^{t+1}\).  For the root \(\beta^{q^i}\), the corresponding condition is
\((\ell_A^{\mathsf T}Wv)^{q^i}\neq0\), since \(A,a_0,W\) are defined over
\(\F_q\). Hence this condition holds simultaneously at all roots of \(p\), which
is equivalent to having one Jordan block at every root.  This proves
\eqref{eq:one-layer-cyclicity-test}.

The \(\F_q\)-linear map
\[
   W\longmapsto\ell_A^{\mathsf T}Wv
   \quad\text{from }\operatorname{Mat}_{dt\times d}(\F_q)
      \text{ to }\F_{q^d}
\]
is surjective.  Indeed, \(a_0v=\beta v\) and \(a_0\) has entries in \(\F_q\), so
the \(\F_q\)-span of the coordinates of \(v\) is stable under multiplication by
\(\beta\).  It is nonzero and \(p\) is irreducible, so it must be all of
\(\F_{q^d}\).  Choose a nonzero coordinate of \(\ell_A\) and vary only the
corresponding row of \(W\).  The value \(\ell_A^{\mathsf T}Wv\) then ranges over
all of \(\F_{q^d}\).  Consequently, exactly a proportion \(q^{-d}\) of the
blocks \(W\) fail the cyclicity test in \eqref{eq:one-layer-cyclicity-test}.

Now choose \(x=A^{-1}B\) with \(A,B\in S_t\).  For a
\(dt\)-by-\(d\) matrix \(W\), put
\[
   A_W:=\begin{pmatrix}A&W\\0&a_0\end{pmatrix},
   \qquad
   B_W:=\begin{pmatrix}B&AZ+W\\0&a_0\end{pmatrix}.
\]
By \eqref{eq:one-layer-cyclicity-test}, the proportion of choices of \(W\)
for which \(A_W\notin S_{t+1}\) or \(B_W\notin S_{t+1}\) is at most
\(
   2q^{-d}\leq\frac12<1.
\)
Choose \(W\) outside this union.  Then \(A_W,B_W\in S_{t+1}\), while
\[
   A_W^{-1}B_W=\begin{pmatrix}x&Z\\0&I_d\end{pmatrix},
\]
which proves the proposition.
\end{proof}

We first apply this proposition to unipotent targets with
\(\rank(u-I)<t\).

\begin{lemma}
\label{lem:general-q-small-unipotents}
Assume \(d\geq2\).  Let \(u\in\GL_{dt}(q)\) be unipotent, and put
\(s=\rank(u-I)\geq1\).  If \(t\geq s+1\), then
\[
   u\in S_t^{-1}S_t.
\]
\end{lemma}

To apply \cref{prop:add-primary-layer} successively, we construct a
decomposition
\[
   \F_q^{dt}=V_1\oplus\cdots\oplus V_t,
   \qquad \dim V_j=d,
\]
such that, for \(N=u-I\),
\[
   N(V_1)=0,
   \qquad N(V_j)\subseteq V_{j-1}
   \quad(2\leq j\leq t).
\]
Then, relative to
\((V_1\oplus\cdots\oplus V_j)\oplus V_{j+1}\), the restriction of \(u\)
has the block form in \cref{prop:add-primary-layer} for every \(j<t\).

\begin{proof}
Put \(N=u-I\).  Choose Jordan chains
\[
   v_{i,1},\ldots,v_{i,b_i}
   \qquad(1\leq i\leq k)
\]
for the nontrivial blocks of \(N\), oriented so that
\[
   Nv_{i,1}=0,
   \qquad
   Nv_{i,j}=v_{i,j-1}
   \quad(2\leq j\leq b_i),
\]
where \(b_i\geq2\) and
\[
   \sum_{i=1}^k(b_i-1)=\rank N=s.
\]
For each \(i\), set
\[
   r_i:=1+\sum_{h<i}(b_h-1),
\]
and, for every \(j\) with \(1\leq j\leq t\), let \(V_j\) initially be the span
of the vectors \(v_{i,h}\) for which \(r_i+h-1=j\).  Since
\(r_{i+1}=r_i+b_i-1\), only the last vector of one chain and the first vector of
the next can lie in the same \(V_j\).  Hence \(\dim V_j\leq2\leq d\), and the
largest index that occurs is
\[
   r_k+b_k-1=1+\sum_{i=1}^k(b_i-1)=s+1\leq t.
\]
The remaining \(dt-(s+k)\) Jordan-basis vectors come from identity blocks
of \(u\) and are killed by \(N\).  Since
\[
   \sum_{j=1}^t\bigl(d-\dim V_j\bigr)=dt-(s+k),
\]
we can use these remaining vectors to enlarge each \(V_j\) to dimension \(d\).
The resulting subspaces form a direct-sum decomposition of \(\F_q^{dt}\), and
the Jordan-chain relations give
\[
   N(V_1)=0,
   \qquad N(V_j)\subseteq V_{j-1}
   \quad(2\leq j\leq t).
\]

For \(1\leq j\leq t\), let \(u_j\) be the matrix of the restriction of \(u\)
to \(V_1\oplus\cdots\oplus V_j\).  We prove inductively that
\[
   u_j\in S_j^{-1}S_j.
\]
The base case is \(u_1=I_d\in S_1^{-1}S_1\).  Suppose the assertion holds
for some \(j<t\).  Relative to
\((V_1\oplus\cdots\oplus V_j)\oplus V_{j+1}\),
\[
   u_{j+1}
      =\begin{pmatrix}
          u_j&Z_j\\
          0&I_d
        \end{pmatrix},
\]
where \(Z_j\) is the matrix of
\(N|_{V_{j+1}}\colon V_{j+1}\to V_1\oplus\cdots\oplus V_j\).  Applying
\cref{prop:add-primary-layer} with \(t=j\), \(x=u_j\), and \(Z=Z_j\) gives
\[
   u_{j+1}\in S_{j+1}^{-1}S_{j+1}.
\]
Induction gives \(u_t\in S_t^{-1}S_t\).  Since \(u_t\) is the matrix of
\(u\) in the chosen basis and the product \(S_t^{-1}S_t\) is conjugacy
invariant, the lemma follows.
\end{proof}

We next cover unipotent targets of fixed support by \(S_1^{-1}S_1\), when
\(\deg p\) is sufficiently large.

\begin{proposition}
\label{prop:general-q-irreducible-difference}
For every fixed \(s\geq2\) there is an integer \(N_0\geq8\), depending only on
\(s\), such that the following holds for every prime power \(q\).
Let \(p\in\F_q[T]\) be monic irreducible of degree \(n\geq N_0\), with
\(p(0)\neq0\), and put \(S_1=C(p)^{\GL_n(q)}\).
If \(u\in\GL_n(q)\) is unipotent with \(\rank(u-I)=s\), then
\[
   u\in S_1^{-1}S_1.
\]
\end{proposition}

\begin{proof}[Proof of Proposition~\ref{prop:general-q-irreducible-difference}]
Put \(a=C(p)\) and \(G=\GL_n(q)\).  Since \(u\) and \(u^{-1}\) have the
same unipotent Jordan type, all character values at \(u\) are real.
Normalize the Frobenius sum by its \(q-1\) linear-character terms:
\[
   F_a(u):=\frac1{q-1}\sum_{\chi\in\Irr(G)}
      |\chi(a)|^2\frac{\chi(u)}{\chi(1)}.
\]
It is enough to prove \(F_a(u)>0\).
The linear characters contribute \(1\).  We will show that the other
terms of cuspidal degree \(1\) are nonnegative, and that the total
contribution from larger cuspidal degrees has absolute value less than
\(2/5\) for \(n\geq N_0\).

Since \(a=C(p)\) has irreducible characteristic polynomial, Green's formula
\cite[Proposition~4.7(i)--(iii)]{LewisReinerStanton2014} says that the only
characters that can be nonzero at \(a\) are the primary hooks
\[
   \chi^{\varphi,(n/e-j,1^j)},
   \qquad e\mid n,\quad 0\leq j<n/e,
\]
where \(\varphi\) is a cuspidal character of \(\GL_e(q)\).  Their values at
\(a\) have absolute value at most \(e\), and exactly \(1\) when \(e=1\).
Let \(E_e(u)\) be the contribution to \(F_a(u)\) from characters of
cuspidal degree \(e\), so that \(F_a(u)=\sum_{e\mid n}E_e(u)\).

\emph{The principal contribution \(e=1\).}
These characters are the \(q-1\) determinant twists of each unipotent
hook character, and every such character satisfies \(|\chi(a)|^2=1\).
Since \(\det u=1\), the twists agree on \(u\) and have the same degree,
so summing over them cancels the normalizing factor \(1/(q-1)\).
Every unipotent irreducible character of \(\GL_n(q)\) takes nonnegative
values on unipotent elements; see
\cite[Proposition~3.4.14(i),(iii)]{Haiman2002}.
Thus the trivial character contributes \(1\), and all remaining terms
are nonnegative:
\[
   E_1(u)=1+\sum_{j=1}^{n-1}
      \frac{\chi_{(n-j,1^j)}(u)}{\chi_{(n-j,1^j)}(1)}
      \geq1.
\]

\emph{Bounded cuspidal degrees \(2\leq e\leq s\).}
There are at most \((q^e-1)/e\) cuspidal characters of \(\GL_e(q)\)
(see \cite[Section~3.4]{LewisReinerStanton2014}),
and \(n/e\) hooks for each.  Hence
\begin{equation}
   \frac1{q-1}\sum_{\substack{\chi\text{ primary hook}\\
                             \text{of cuspidal degree }e}}
      |\chi(a)|^2
      \leq e^2\cdot\frac ne\cdot\frac{q^e-1}{e(q-1)}
      \leq2nq^{e-1}.
   \label{eq:irreducible-sector-weight}
\end{equation}
For \(e\geq2\), these primary characters have level \(n\), as observed
in \cref{sec:characters-green}.  Hence \cref{thm:level-dimension-bounds} gives
\begin{equation}
   \chi(1)>q^{n^2/4-2}\geq q^{n^2/5}
   \qquad(n\geq8).
   \label{eq:primary-degree-lower-bound}
\end{equation}

Let \(\sigma>0\) be as in \cref{thm:larsen-tiep-character-ratio}.
Since \(u\in\SL_n(q)\), \eqref{eq:LT-ratio-GL} gives
\[
   \frac{|\chi(u)|}{\chi(1)}
      \leq\chi(1)^{-\sigma s/n}
      \leq q^{-\sigma sn/5}
\]
uniformly in \(q\) and \(e\geq2\).
Combining this with \eqref{eq:irreducible-sector-weight} gives
\begin{equation}
   \sum_{\substack{e\mid n\\2\leq e\leq s}}|E_e(u)|
      \leq2snq^{s-1-\sigma sn/5}<\frac15
   \label{eq:bounded-cuspidal-error}
\end{equation}
for all \(n\geq N_0\), after taking \(N_0\) sufficiently large depending
only on the fixed \(s\).  

\emph{All remaining cuspidal degrees \(e>s\).}
By \cref{lem:primary-exact-ratio}, every character in these ranges has
the same normalized value
\[
   \frac{\chi(u)}{\chi(1)}
      =\frac{(-1)^s}{\prod_{i=n-s}^{n-1}(q^i-1)}.
\]
We have
\(\sum_{\chi\in\Irr(G)}|\chi(a)|^2=|C_G(a)|=q^n-1\).  Therefore,
after increasing \(N_0\) if necessary,
\begin{equation}
   \begin{aligned}
   \left|\sum_{\substack{e\mid n\\e>s}}E_e(u)\right|
      &\leq
      \frac{q^n-1}{(q-1)\prod_{i=n-s}^{n-1}(q^i-1)}\\
      &\leq4 \cdot 
         q^{-(s-1)n+s(s+1)/2-1}
       <\frac15,
   \end{aligned}
   \label{eq:large-cuspidal-error}
\end{equation}
for large enough $n$, since \(s\geq2\).  

Together, the three ranges give \(F_a(u)>1-\frac15-\frac15=\frac35>0\) once
\(n\geq N_0\).  The Frobenius count is therefore positive, as required.
\end{proof}

\begin{proof}[Proof of Theorem~\ref{thm:primary-cyclic-difference}]
Let \(B\) be supplied by \cref{prop:constant-support-reduction}.
For each integer \(s\) with \(2\leq s<B\),
\cref{prop:general-q-irreducible-difference} supplies a threshold depending
only on \(s\).  Choose \(D\geq2B\) at least as large as all these finitely
many thresholds, so that the proposition applies throughout this range
whenever \(d\geq D\).  Choose \(N\) at least
\(8\), \(2B\), \(3(B-1)\), and \((B-1)D\).

We first prove \eqref{eq:primary-difference-cover}.  Let \(x\in\SL_n(q)\)
represent a target in \(\PSL_n(q)\).  If \(\supp(x)\geq B\), then
\cref{prop:constant-support-reduction} gives \(x\in S_r^{-1}S_r\).  We may
therefore assume that \(\supp(x)<B\).  If \(x\) is nonprimary, the theorem of
B\"unger and Nielsen \cite[Theorem~4.2]{BungerNielsen1999} gives
\(x\in S_r^{-1}S_r\).  For a short proof, see
\cite[Lemma~1.6(1) and \S4]{Nielsen2025}.  We may therefore assume that \(x\) is
primary of degree \(e\).  If \(e\geq2\), every eigenspace of \(x\) over
\(\overline{\F}_q\) has dimension at most \(n/e\), and hence
\[
   \supp(x)\geq n-\frac ne\geq\frac n2\geq B,
\]
contrary to our assumption.  Thus \(e=1\), so \(x=zu\) for some
\(z\in\F_q^\times\) and some unipotent \(u\). \(x\) and \(u\) have the
same image in \(\PSL_n(q)\), and it suffices to cover \(u\).

For a unipotent matrix, \(\supp(u)=\rank(u-I)\).  The case of rank zero is the
identity, while rank one is excluded from the first assertion.  It remains to
consider
\[
   2\leq s:=\rank(u-I)<B.
\]
If \(d=1\), then \(p=T-\alpha\) for some \(\alpha\in\F_q^\times\), and
\(S_r\) is a scalar multiple of the conjugacy class of \(J_n(1)\).  The scalar
cancels in \(S_r^{-1}S_r\), while
\(r=n\geq3(B-1)\geq3s\). Hence we can use
\cref{lem:regular-unipotent-difference} in this case.

Assume now that \(d\geq2\).  If \(r\geq s+1\), then
\cref{lem:general-q-small-unipotents}, with \(t=r\), gives
\(u\in S_r^{-1}S_r\).  It remains to treat \(r\leq s\).  Since
\(r\leq s\leq B-1\),
\[
   d=\frac nr\geq\frac{N}{B-1}\geq D.
\]
The direct sum of the nontrivial Jordan blocks of \(u\) has dimension at most
\(2s\): its dimension is \(s\) plus the number of these blocks, and their
number is at most \(s\).  Since \(d\geq D\geq2B>2s\), adjoining identity
blocks gives a unipotent matrix \(u_0\in\GL_d(q)\) such that
\[
   u\sim u_0\oplus I_{d(r-1)},
   \qquad \rank(u_0-I)=s.
\]
By \cref{prop:general-q-irreducible-difference}, applied in dimension \(d\),
\(u_0\in S_1^{-1}S_1\).  If \(r=1\), this is the desired conclusion.  If
\(r\geq2\), applying \cref{prop:add-primary-layer} \(r-1\) times with
\(Z=0\) gives
\[
   u_0\oplus I_{d(r-1)}\in S_r^{-1}S_r,
\]
and the displayed conjugacy therefore gives \(u\in S_r^{-1}S_r\).

We have now covered every target whose support is different from one.
Projecting these factorizations to \(\PGL_n(q)\) proves
\eqref{eq:primary-difference-cover}.

For \eqref{eq:primary-difference-all}, assume \(r\geq2\). Only targets of
support one remain. The nonprimary case was handled above, so we may assume
that the target is primary. It must have degree one, since one of degree at
least two has support at least \(n/2\). It therefore has a unipotent
representative \(u\) with
\(\rank(u-I)=1\).  If \(d=1\), then
\(r=n\geq N\geq3\), and \cref{lem:regular-unipotent-difference} applies.  If
\(d\geq2\), then \cref{lem:general-q-small-unipotents} applies with \(s=1\)
and \(t=r\).  Thus every target is covered when \(r\geq2\). 
\end{proof}

\section{Arbitrary large binary classes}
\label{sec:binary-large-classes}

In this section we prove our main theorem \cref{thm:main-binary}, starting with
targets whose support is a fixed positive fraction of \(n\).

\begin{proposition}
\label{prop:binary-fixed-fraction-mixing}
For every \(c>0\), there are \(\epsilon>0\) and \(N \in \mathbb{N}\) such that, whenever
\(n>N\), the following holds.  Let \(S\subseteq\GL_n(2)\) be a conjugacy
class, and let \(x\in\GL_n(2)\).  If
\[
   |S|\geq|\GL_n(2)|^{1-\epsilon},
   \qquad
   \supp(x)\geq cn,
\]
then
\[
   x\in SS^{-1}.
\]
\end{proposition}

\begin{proof}
Let \(\sigma\) be as in \cref{thm:larsen-tiep-character-ratio}.  Applying
\cite[Theorem~1.3]{GLTClassical} with exponent \(\sigma c/4\), choose
\(\epsilon>0\), depending only on \(c\), so that
\[
   |C_{\GL_n(2)}(a)|\leq|\GL_n(2)|^\epsilon
   \quad\Longrightarrow\quad
   |\chi(a)|\leq\chi(1)^{\sigma c/4}
\]
for every \(a\in\GL_n(2)\) and \(\chi\in\Irr(\GL_n(2))\).  Choose
\(a\in S\). By the size assumption, the displayed bound applies to \(a\).
\cref{thm:larsen-tiep-character-ratio} gives, for every nonlinear \(\chi\),
\[
   \frac{|\chi(x)|}{\chi(1)}\leq\chi(1)^{-\sigma c}.
\]

The trivial character contributes one to the Frobenius sum.  By the two
character bounds above, the remaining contribution has absolute value at most
\[
   \sum_{\chi(1)>1}|\chi(a)|^2
      \frac{|\chi(x)|}{\chi(1)}
   \leq\sum_{\chi(1)>1}\chi(1)^{-\sigma c/2}.
\]
For a finite group \(G\), its \emph{Witten zeta function} is
\[
   \zeta_G(t):=\sum_{\chi\in\Irr(G)}\chi(1)^{-t},\qquad t>0.
\]
For \(n\geq3\), the sum over nonlinear characters above equals
\(\zeta_{\GL_n(2)}(\sigma c/2)-1\), which tends to zero as \(n\to\infty\) by
Liebeck and Shalev \cite[Theorem~1.2]{LiebeckShalev2005}. Choose large enough
\(N\geq5\), so the sum cannot cancel the trivial term, and the Frobenius
criterion proves the claim.
\end{proof}

We use the following defintion and lemma to estimate the size of a conjugacy
class:

\begin{definition}
For a direct sum \(D=\bigoplus_iD_i\in\GL_m(2)\) of elementary-divisor
blocks, define
\[
   E(D):=\dim_{\F_2}C_{\operatorname{Mat}_m(2)}(D)-m.
\]
For example, if \(\deg p=d\), the matrix-centralizer algebra of
\(C(p)^{\oplus b}\) has dimension \(db^2\).  Deleting one copy changes this
to \(d(b-1)^2\), a decrease of \(d(2b-1)\): the diagonal block contributes
\(d\), and its off-diagonal row and column contribute \(2d(b-1)\).

For general \(D\), the possible diagonal blocks corresponding to \(D_i\)
form its matrix centralizer \(\F_2[D_i]\), a space of dimension \(\dim D_i\).
Thus the diagonal blocks contribute \(m\), and \(E(D)\) is exactly the total
dimension of the off-diagonal block spaces.
\end{definition}

\begin{lemma}\label{lem:binary-centralizer-size}
For every such \(D\),
\[
   2^{E(D)}\leq |C_{\GL_m(2)}(D)|\leq2^{E(D)+m}.
\]
\end{lemma}

\begin{proof}
Apply \eqref{eq:general-centralizer-proportion} with \(q=2\).
Each factor is at least \(1/2\), and \(\sum_{p,r}b_{p,r}\leq m\), so this
proportion is at least \(2^{-m}\).  Since
\(|C_{\operatorname{Mat}_m(2)}(D)|=2^{E(D)+m}\), this gives the lower
bound.  The upper bound follows from
\(C_{\GL_m(2)}(D)\subseteq C_{\operatorname{Mat}_m(2)}(D)\).
\end{proof}

We now prove \cref{thm:main-binary}. We begin by expressing an element of \(S\)
as a direct sum of rational-canonical blocks and cancelling identical blocks in
the two factors.  Cancelling a block \(D\), means using \(D\) in both factors,
so that its contribution to the difference product is \(DD^{-1}=I\). If a single
source block is large enough to contain the nonidentity part of the target, we
apply \cref{thm:primary-cyclic-difference} to that block and cancel the
remaining parts. Otherwise, we cancel blocks until the target's support reaches
at least one-quarter of the remaining dimension, at which point
\cref{prop:binary-fixed-fraction-mixing} applies.

\begin{proof}[Proof of \cref{thm:main-binary}]
Let \(N\) be the dimension threshold in
\cref{thm:primary-cyclic-difference}.  Apply
\cref{prop:binary-fixed-fraction-mixing} with \(c=1/4\), and denote its
constants by \(\epsilon_1\) and \(N_1\).  Choose an integer \(B\) such that
\(2B\geq N\) and every \(m>2B\) satisfies
\[
   m>N_1,
   \qquad
   2^{m+(\epsilon_1/2)m^2}\leq|\GL_m(2)|^{\epsilon_1}.
\]
There are only finitely many invertible elementary-divisor blocks over
\(\F_2\) of dimension less than \(2B\).  Let \(L\) be their number.  Choose
\[
   0<\epsilon<\min\left\{\frac{\epsilon_1}{2},\frac1{8B^2L}\right\}
\]
and an integer \(N_0>\max\{N_1,8B^2L\}\).

Fix \(n\geq N_0\), and let \(x\in\GL_n(2)\) have
\(s=\supp(x)\neq1\).  The identity is automatic, so assume \(s\geq2\).
Write
\[
   S=a^{\GL_n(2)},
   \qquad
   a=\bigoplus_i C(p_i^{r_i}).
\]
The class-size hypothesis and \cref{lem:binary-centralizer-size} give
\[
   2^{E(a)}
      \leq |C_{\GL_n(2)}(a)|
      =\frac{|\GL_n(2)|}{|S|}
      \leq|\GL_n(2)|^\epsilon
      <2^{\epsilon n^2}.
\]
Thus \(E(a)/n^2\leq\epsilon\).
If \(s\geq n/4\), \cref{prop:binary-fixed-fraction-mixing} applied in
dimension \(n\) proves the result.

Assume that \(s<n/4\).  Then \(s=\rank(x-I)\).  After conjugation, write
\[
   x=x_0\oplus I,
   \qquad
   \dim x_0\leq2s.
\]

If some elementary-divisor block of $a$ has dimension at least
\(2\max\{s,B\}\), then \(x_0\) fits on that block. The padded target still has
support \(s\), so \cref{thm:primary-cyclic-difference} covers it.  Taking the
remaining blocks identically in the two elements of \(S\) makes them cancel
in \(SS^{-1}\), giving the identity on the complementary summand and
completing the factorization of \(x\).

Suppose now that every block has dimension less than
\(2\max\{s,B\}\).  We first show that \(s\geq B\).  Otherwise all blocks
have dimension less than \(2B\).  If \(b_1,\ldots,b_L\) are the
multiplicities of the possible block types and \(b=\sum b_i\), then
\(b>n/(2B)\).  Maps between equal blocks give
\[
   E(a)\geq\sum_{i=1}^L b_i(b_i-1)
      \geq\frac{b^2}{L}-b
      >\frac{n^2}{4B^2L}-n
      >\frac{n^2}{8B^2L}
      >\epsilon n^2,
\]
contrary to \(E(a)\leq\epsilon n^2\).  Hence \(s\geq B\), and every
block has dimension less than \(2s\).

We now repeatedly cancel source blocks against the identity part of the
target.  At a stage with \(a=\bigoplus_i a_i\in\GL_m(2)\), \(m>4s\), put
\(m_i=\dim a_i\).  Since \(m-\dim x_0>2s>m_i\), we may take \(a_i\)
identically in the two factors, producing \(I_{m_i}\).  Thus the target's
support remains \(s\) while its ambient dimension falls to \(m-m_i\), so
its relative support increases from \(s/m\) to \(s/(m-m_i)\).  We choose
\(a_i\) so that \(E(a)/m^2\) does not increase.

Let \(\Delta_i\) be the decrease in \(E(a)\) caused by removing \(a_i\).
Every off-diagonal block space has two distinct indices, so it contributes
to exactly two of the numbers \(\Delta_i\).  Hence
\[
   \sum_i\Delta_i=2E(a),
   \qquad
   \frac{2E(a)}{m}
      =\sum_i\frac{m_i}{m}\frac{\Delta_i}{m_i}.
\]
The right-hand side is a weighted average of the numbers
\(\Delta_i/m_i\).  Thus some \(i\) satisfies
\(\Delta_i\geq2m_iE(a)/m\).  If \(a'\) is obtained by removing this block,
then
\[
   E(a')=E(a)-\Delta_i
      \leq E(a)\left(1-\frac{2m_i}{m}\right)
      \leq E(a)\left(1-\frac{m_i}{m}\right)^2,
\]
and, since \(\dim a'=m-m_i\),
\[
   \frac{E(a')}{(\dim a')^2}\leq\frac{E(a)}{m^2}.
\]

Relabel \(a'\) as \(a\) and repeat whenever the current dimension exceeds
\(4s\).  Each step preserves the normalized bound above.  When the process
stops, the remaining summand \(a_0\in\GL_m(2)\) has dimension
\(m\leq4s\).  The preceding dimension was greater than \(4s\), while the
last deleted block had dimension less than \(2s\), so \(m>2s\).  Therefore
\begin{equation}
   2s<m\leq4s,
   \qquad
   \frac{E(a_0)}{m^2}\leq\epsilon.
   \label{eq:binary-peeled-source}
\end{equation}
In particular, the padded target \(x_0\oplus I_{m-\dim x_0}\) is defined
and has support \(s\geq m/4\).

Since \(m>2s\geq2B\),
\Cref{lem:binary-centralizer-size} and \eqref{eq:binary-peeled-source} give
\[
   |C_{\GL_m(2)}(a_0)|
      \leq2^{m+\epsilon m^2}
      \leq2^{m+(\epsilon_1/2)m^2}
      \leq|\GL_m(2)|^{\epsilon_1}.
\]
Thus \cref{prop:binary-fixed-fraction-mixing} applies in dimension \(m\) to
the class of \(a_0\) and the padded target.  Cancelling the deleted blocks
gives \(x\in SS^{-1}\), proving \cref{thm:main-binary}.
\end{proof}

The next two
sections show, respectively, why the theorem uses \(SS^{-1}\) rather than
\(S^2\), and why its conclusion for arbitrary source classes does not
extend to all finite fields. The latter failure is demonstrated in the
simple groups \(\PSL_n(8)\).

\section{Repeated-block failure of square covering over
\texorpdfstring{$\F_2$}{F2}}
\label{sec:f2-square-obstruction}

To build the counterexample we first need the following lemma:

\begin{lemma}
\label{lem:polynomial-rank}
Let \(U,V\in\operatorname{Mat}_n(F)\), where \(F\) is a field, and let
\(f\in F[T]\) have degree \(e\).  Then
\[
   \rank\bigl(f(U)-f(V)\bigr)
   \leq e\,\rank(U-V).
\]
\end{lemma}

\begin{proof}
   Write \(f(T)=\sum_{i=0}^e a_iT^i\).
For \(i\geq1\), the telescoping identity
\[
   U^i-V^i
   =\sum_{j=0}^{i-1}U^{i-1-j}(U-V)V^j
\]
gives, after grouping terms with the same \(j\),
\[
   f(U)-f(V)
   =\sum_{j=0}^{e-1}
     \left(\sum_{i=j+1}^{e}a_iU^{i-1-j}\right)(U-V)V^j.
\]
Each of these \(e\) summands has rank at most \(\rank(U-V)\), so
subadditivity of rank proves the claim.
\end{proof}

\Cref{thm:intro-f2-obstruction} follows from the following theorem and the
class-size computation after its proof.

\begin{theorem}
\label{thm:f2-repeated-block-obstruction}
Let \(p\in\F_2[T]\) be monic irreducible of degree \(d\geq3\), and assume that
\(p\neq p^\vee\).  Let \(m\geq1\), put \(n=dm\), and set
\(S:=\bigl(C(p)^{\oplus m}\bigr)^{\GL_n(2)}\).  Then \(S\) satisfies
\[
   S^2\cap
   \{x\in\GL_n(2):\supp(x)<m\}
   =\varnothing.
\]
\end{theorem}

\begin{proof}
Inversion sends \(S\) to the class of \(C(p^\vee)^{\oplus m}\).  Since
\(p\neq p^\vee\), the two classes have different characteristic polynomials.
Thus \(S\neq S^{-1}\).

Write \(x=XY\) with \(X,Y\in S\), and put \(Z=Y^{-1}\).  Then
\(p(X)=p^\vee(Z)=0\).  Since \(p\) and \(p^\vee\) are coprime, choose
\(a,b\in\F_2[T]\) with \(ap+bp^\vee=1\).  Evaluating at \(Z\) gives
\(a(Z)p(Z)=I\), so \(p(Z)\) is invertible.  Moreover, \(X=xZ\) and
\(Z-X=(I-x)Z\), so
\cref{lem:polynomial-rank} gives
\[
   dm=n
   =\rank p(Z)
   =\rank\bigl(p(Z)-p(X)\bigr)
   \leq d\,\rank(Z-X)
   =d\,\rank(x-I).
\]
Thus \(\rank(x-I)\geq m\).  Since \(d\geq2\), we have \(m\leq n/2\).
Suppose that \(\supp(x)<m\).  A scalar attaining the projective support then
has an eigenspace of dimension greater than \(n/2\).  Its Frobenius conjugate
has the same multiplicity.  Hence the scalar is Frobenius-fixed, since two
distinct eigenspaces cannot both have dimension greater than \(n/2\).  Over
\(\F_2\), the only nonzero Frobenius-fixed scalar is \(1\).  Therefore
\(\supp(x)=\rank(x-I)<m\), contradicting the rank bound.
\end{proof}

The repeated-block centralizer formula
\eqref{eq:repeated-block-centralizer} gives, uniformly in \(d\) and \(m\),
\begin{align*}
   \log_2|S|&=n^2-dm^2+O(1),\\
   \frac{\log|S|}{\log|\GL_n(2)|}
      &=1-\frac1d+O(n^{-2}).
\end{align*}

\section{Proof of~\texorpdfstring{\cref{thm:intro-f8-obstruction}}{Theorem
\ref{thm:intro-f8-obstruction}}}
\label{sec:f8-obstruction}

We show that the covering conclusion of~\cref{thm:main-binary}
fails when moving to larger fields.

Let \(d\geq4\) be even and \(m\geq3\), with \(7\nmid dm\), and put
\(n=dm\). Fix \(\lambda\in\F_8^\times\setminus\{1\}\).
Since \(7\nmid n\), the center of \(\SL_n(8)\) is trivial, and we identify
\(G:=\SL_n(8)=\PSL_n(8)\). Moreover, every \(P\in\GL_n(8)\) can be
multiplied by a scalar to obtain determinant one, since the \(n\)-th power
map on \(\F_8^\times\) is bijective. Thus every \(\GL_n(8)\)-conjugacy
class contained in \(G\) remains a single \(G\)-conjugacy class.

Choose \(\alpha\in\F_{8^d}^\times\) of order \(8^{d/2}+1\), and let \(p\) be its
minimal polynomial over \(\F_8\). The degree of \(p\) is the least positive
\(k\) such that \(\alpha^{8^k}=\alpha\), so \(\deg p=d\). Since
\(\alpha^{-1}=\alpha^{8^{d/2}}\), we have \(p=p^\vee\), and
\(
   N_{\F_{8^d}/\F_8}(\alpha)
   =1.
\)
Consequently \(\det C(p)=p(0)=1\). Set
\[
   a:=C(p)^{\oplus m},\qquad S:=a^G=a^{\GL_n(8)}.
\]
\Cref{thm:intro-f8-obstruction} follows from the next rank statement and
the class-size computation after its proof.

\begin{theorem}
\label{thm:f8-scalar-obstruction}
Let \( a:=C(p)^{\oplus m}, S:=a^G=a^{\GL_n(8)}. \). Then \(S^2=SS^{-1}\), and
every \(x\in S^2\) satisfies for every \(\lambda\in\F_8^\times\setminus\{1\}\),
\[
   \rank(x-\lambda I)\geq m.
\]

So, for each integer \(s\) with \(2\leq s<m\), there is an element
\(x_s\in G\) with \(\supp(x_s)=s\) and \(x_s\notin S^2\).
\end{theorem}

\begin{proof}
Since \(p=p^\vee\) and the \(\GL_n(8)\)- and \(G\)-classes coincide,
we have \(S=S^{-1}\).

The irreducible polynomials \(p(T)\) and \(p(\lambda T)\) are coprime.
Otherwise the monic polynomials \(p(T)\) and
\(\lambda^{-d}p(\lambda T)\) would be equal, and comparing their
constant terms would give \(\lambda^d=1\), contrary to
\(7\nmid d\).

\(p(\lambda Y)\) is invertible when \(p(Y)=0\). Indeed, by B\'ezout's
identity, we may choose \(u,v\in\F_8[T]\) with
\(u(T)p(T)+v(T)p(\lambda T)=1\). Evaluating at \(Y\) with
\(p(Y)=0\) gives \(v(Y)p(\lambda Y)=I\).

Write \(x=XY^{-1}\), where \(X,Y\in S\). Then \(X=xY\) and
\(p(X)=p(Y)=0\). Applying \cref{lem:polynomial-rank} gives
\[
   n=\rank\bigl(p(\lambda Y)-p(X)\bigr)
   \leq d\,\rank(\lambda Y-X)
   =d\,\rank(x-\lambda I).
\]
Since \(n=dm\), this proves the rank bound.

For completeness, we construct \(x_s\) for each \(2\le s<m\). Choose
\(\mu\in\F_8^\times\setminus\{\lambda,\lambda^{-(n-s+1)}\}\), put
\(\nu:=\lambda^{-(n-s)}\mu^{-1}\), and define
\[
   x_s:=\operatorname{diag}(\lambda I_{n-s},\mu,\nu,I_{s-2}).
\]
 We have \(\mu,\nu\neq\lambda\) and \(\det x_s=1\).
Its \(\lambda\)-eigenspace has dimension \(n-s>n/2\), so
\(\supp(x_s)=\rank(x_s-\lambda I)=s<m\). The rank bound therefore
gives \(x_s\notin S^2\).
\end{proof}

Finally, \(C_{\GL_n(8)}(a)\cong\GL_m(8^d)\). Since the conjugacy
classes in \(G\) and \(\GL_n(8)\) coincide,
\[
   |S|=\frac{|\GL_n(8)|}{|\GL_m(8^d)|},\qquad
   \frac{\log|S|}{\log|G|}=1-\frac1d+O(n^{-2}),
\]
uniformly in \(d\) and \(m\), by the standard product formula
\(\log_8|\GL_k(8^e)|=ek^2+O(1)\). This completes the proof of
\cref{thm:intro-f8-obstruction}.

\section{Dense normal subsets and Pyber's conjecture}
\label{sec:dense-normal-subsets}

We conclude with a related obstruction for unions of conjugacy classes.
A subset of a group is \emph{normal} if it is invariant under conjugation.
Pyber's conjecture in Problem~20.74(a) of the Kourovka Notebook
\cite{Kourovka2026} can be stated as follows.

\begin{conjecture}[Pyber]
There are only finitely many nonabelian finite simple groups \(G\) that
admit an inverse-closed normal subset \(S\subseteq G\) such that
\[
  |S|>\frac{|G|}{\log_2|G|}
  \qquad\text{and}\qquad
  S^2\neq G.
\]
\end{conjecture}

The following explicit family disproves the conjecture.  We give
the construction over \(\F_5\), where the density estimate is particularly
simple.

\begin{theorem}
\label{thm:pyber-q5}
For every odd \(n\geq5\), identify
\[
  G_n:=\SL_n(5)=\PSL_n(5)
\]
and define
\[
  S:=
  \{A\in\SL_n(5):\dim\ker(A-I)\geq2
  \text{ and }A+I\text{ is invertible}\}.
\]
Then \(S\) is normal and inverse-closed,
\[
  \frac{|S|}{|G_n|}\geq\frac1{960},
  \qquad
  \operatorname{diag}(1,-I_{n-1})\notin S^2.
\]
\end{theorem}

\begin{proof}
The two conditions defining \(S\)
are invariant under conjugation.  They are also invariant under
inversion, since
\[
  A^{-1}-I=-A^{-1}(A-I),
  \qquad
  A^{-1}+I=A^{-1}(A+I).
\]
Thus \(S\) is normal and inverse-closed.

Every \(g\in S^2\) can therefore be written as \(g=AB^{-1}\), with
\(A,B\in S\).  On \(U=\ker(A-I)\) one has
\[
  (A+B)v=(I+B)v.
\]
As \(I+B\) is invertible, \(A+B\) is injective on \(U\).  It follows that
\[
  \rank(g+I)
  =\rank\bigl((A+B)B^{-1}\bigr)
  =\rank(A+B)
  \geq\dim U
  \geq2.
\]
On the other hand, for \(h_n=\operatorname{diag}(1,-I_{n-1})\),
\[
  \det(h_n)=(-1)^{n-1}=1,
  \qquad
  \rank(h_n+I)=1.
\]
Hence \(h_n\in G_n\setminus\{1\}\) but \(h_n\notin S^2\).

It remains to verify the density.  Put \(m=n-2\geq3\) and
\[
  \mathcal R_m:=
  \{D\in\SL_m(5):D-I\text{ and }D+I\text{ are invertible}\}.
\]
If \(D\) is uniform in \(\SL_m(5)\) and
\(\lambda\in\{1,-1\}\), let \(N_\lambda(D)\) be the number of
\(\lambda\)-eigenlines of \(D\).  The group \(\SL_m(5)\) acts
transitively on the nonzero vectors of \(\F_5^m\).  Hence, for each
of the \((5^m-1)/4\) lines, the probability that \(D\) acts on it as
\(\lambda\) is \(1/(5^m-1)\), so
\(\mathbb E N_\lambda(D)=1/4\).  Since \(D-\lambda I\) is singular
exactly when \(N_\lambda(D)\geq1\), Markov's inequality and the union
bound give
\[
  \frac{|\mathcal R_m|}{|\SL_m(5)|}
  \geq1-\sum_{\lambda\in\{1,-1\}}
    \mathbb P(N_\lambda(D)\geq1)
  \geq1-\sum_{\lambda\in\{1,-1\}}
    \mathbb E N_\lambda(D)
  =\frac12.
\]

Now count the matrices \(A=I_U\oplus D\), where
\(\F_5^n=U\oplus W\), \(\dim U=2\), and
\(D\in\mathcal R_m\) acts on \(W\).  The number of ordered
decompositions \(\F_5^n=U\oplus W\), with
\(\dim U=2\) and \(\dim W=m\), is
\[
  \frac{|\GL_n(5)|}{|\GL_2(5)|\,|\GL_m(5)|}.
\]
For each such decomposition, there are \(|\mathcal R_m|\) choices for
\(D\).  Every resulting matrix belongs to \(S\).  Moreover,
\(U=\ker(A-I)\) and \(W=\operatorname{im}(A-I)\), so no matrix is
counted twice.
Using \(|\GL_r(5)|=4|\SL_r(5)|\) and
\(|\GL_2(5)|=480\), we obtain
\[
  \frac{|S|}{|G_n|}
  \geq
  \frac1{|\GL_2(5)|}
  \frac{|\mathcal R_m|}{|\SL_m(5)|}
  \geq\frac1{960}.
\]
Finally, \(|G_n|\to\infty\), so
\(1/960>1/\log_2|G_n|\) for all sufficiently large odd \(n\).
This proves the asserted counterexamples.
\end{proof}

\cref{thm:pyber-q5} also disproves the equal-set covering assertion in
Shalev's 2023 Question~1.5~\cite{ShalevEMS2023}: even an inverse-closed
normal subset of density at least \(1/960\) can have a square missing
a nonidentity element in arbitrarily large finite simple groups.

\begin{remark}
The construction extends to every fixed prime power \(q\geq3\), in
dimensions \(n\geq5\) with \(\gcd(n,q-1)=1\).  (For even \(q\),
replace the exclusion of the eigenvalue \(-1\) by exclusions of
\(\lambda\) and \(\lambda^{-1}\), where
\(\lambda\in\F_q^\times\setminus\{1\}\).)
The resulting inverse-closed normal subsets of \(\PSL_n(q)\) have
density bounded below by a positive constant depending only on \(q\),
while their squares miss a nonidentity element.  We chose \(q=5\)
for simplicity of the size analysis.
\end{remark}

\section*{Acknowledgments}
The author used OpenAI's ChatGPT and Codex, and Google's Gemini during the research and preparation of this manuscript. A first version of the proof of the main theorem was obtained through interactive sessions with ChatGPT 5.6 and Gemini 3.1. In addition, the counterexamples were constructed entirely by these AI systems. Later, ChatGPT Astra simplified the argument in Section 3.2, which originally only bounded the character ratio. This work was partially done while the author was visiting the Simons Institute for the Theory of Computing.

\bibliographystyle{alpha}
\bibliography{references}

\end{document}